\documentclass[reqno,11pt]{amsart}

\usepackage{amsmath,amsfonts,amssymb,amsthm,epsfig,amscd,}
\usepackage[]{fontenc}
\usepackage{enumerate}
\usepackage[cmtip,all]{xy}
\usepackage{tabularx,multirow}

\newtheorem{theorem}{Theorem}[section]
\newtheorem{lemma}[theorem]{Lemma}
\newtheorem{proposition}[theorem]{Proposition}
\newtheorem{conjecture}[theorem]{Conjecture}

\theoremstyle{definition}
\newtheorem{definition}[theorem]{Definition}
\newtheorem{example}[theorem]{Example}
\newtheorem*{notation}{Notation}

\theoremstyle{remark}
\newtheorem{remark}[theorem]{Remark}

 \allowdisplaybreaks
\numberwithin{equation}{section}

\newcommand{\gon}{\mathrm{gon}}
\newcommand{\degree}{\mathrm{d}}

\newcommand{\Div}{\mathrm{Div}}
\newcommand{\Tr}{\mathrm{Tr}}
\newcommand{\Image}{\mathrm{Im}}

\begin{document}

\title{The gonality theorem of Noether for hypersurfaces}


\author{F. Bastianelli}
\address{Dipartimento di Matematica e Applicazioni, Universit\`a degli Studi di Milano-Bicocca, via Cozzi 53, 20125 Milano - Italy}
\email{francesco.bastianelli@unimib.it}
\thanks{The first author has been partially supported by 1) Istituto Nazionale di Alta Matematica ``F. Severi''; 2) FAR 2010 (PV) \emph{``Variet\`a algebriche, calcolo algebrico, grafi orientati e topologici''}; 3) INdAM (GNSAGA)}

\author{R. Cortini}
\address{I. T. G. Fazzini-Mercantini, via Salvo D'Acquisto 30, 63013 Grottammare (AP) - Italy}
\email{renzacortini@fazzinimercantini.it}

\author{P. De Poi}
\address{Dipartimento di Matematica e Informatica, Universit\`a degli Studi di Udine, via delle Scienze 206, 33100 Udine - Italy}
\email{pietro.depoi@uniud.it}
\thanks{The third author has been partially supported by MIUR, project \emph{``Geometria delle variet\`a algebriche e dei loro spazi
di moduli''}}

\date{}


\begin{abstract}
It is well known since Noether that the gonality of a smooth curve ${C\subset \mathbb{P}^2}$ of degree $d\geq 4$ is $d-1$.
Given a $k$-dimensional complex projective variety $X$, the most natural extension of gonality is probably the degree of irrationality, that is the minimum degree of a dominant rational map ${X\dashrightarrow\mathbb{P}^k}$.
In this paper we are aimed at extending the assertion on plane curves to smooth hypersurfaces in $\mathbb{P}^n$ in terms of degree of irrationality.
We prove that both surfaces in $\mathbb{P}^3$ and threefolds in $\mathbb{P}^4$ of sufficiently large degree $d$ have degree of irrationality $d-1$, except for finitely many cases we classify, whose degree of irrationality is $d-2$.
To this aim we use Mumford's technique of induced differentials and we shift the problem to study first order congruences of lines of $\mathbb{P}^n$.
In particular, we also slightly improve the description of such congruences in $\mathbb{P}^4$ and we provide a bound on degree of irrationality of hypersurfaces of arbitrary dimension.
\end{abstract}

\maketitle

\section{Introduction}

Let $C$ be a smooth complex projective curve of genus $g$.
The \emph{gonality} of $C$---denoted by $\gon(C)$---is the important birational invariant defined as the minimum degree $m$ of a covering ${C\longrightarrow \mathbb{P}^1}$.
In the case of plane curves, the gonality is governed by the following classical result stated by Max Noether (cf. \cite{N}).
\begin{theorem}\label{theorem NOETHER}
Let ${C\subset \mathbb{P}^2}$ be a smooth curve of degree ${d\geq 4}$.
Then the gonality of $C$ is ${\gon(C)=d-1}$.
Moreover, any morphism ${C\longrightarrow \mathbb{P}^1}$ of degree $d-1$ is obtained projecting $C$ from one of its points.
\end{theorem}
As it has been pointed out in \cite{Ci} and \cite{H2}, Noether's proof of the latter result contains a gap, but the correctness of the assertion is now proved in several different ways (cf. \cite{Ci, H2, Na} and also Remark \ref{remark PROOF OF NOETHER}).

On the other hand, the problem of establishing the gonality of a nonsingular curve contained in a larger projective space is still open.
However analogous results have been proved under further assumptions (see e.g. \cite{Bsl, Fk, HS, Mt}), and the gonality of the considered space curves is computed by projections from suitable multisecant lines.

In this paper we are instead aimed at extending Noether's theorem to nonsingular hypersurfaces of $\mathbb{P}^n$.
In particular, we prove analogous assertions for surfaces and threefolds.

\smallskip
Dealing with higher dimensional varieties, there are diverse attempts at extending the notion of gonality (cf. \cite{B} and \cite{Ko}).
In particular, the most natural one in the case of subvarieties of $\mathbb{P}^n$ is probably the so-called degree of irrationality.
Given a complex irreducible projective variety $X$ of dimension $k$, the \emph{degree of irrationality} of $X$ is defined as the integer
\begin{displaymath}
\degree_r(X):=\min \left\{m\in \mathbb{N} \left| \begin{array}{l}\textrm{there exists a dominant rational }\\ \textrm{map } X\dashrightarrow \mathbb{P}^k \textrm{ of degree } m\end{array}\right.\right\}.
\end{displaymath}
We remark that it does coincide with gonality in the case of curves.
Moreover, the degree of irrationality is still a birational invariant, and having $\degree_r(X)=1$ is equivalent to being a rational variety.
We would also like to recall that the degree of irrationality has been initially defined in an algebraic context in terms of fields extensions (\cite{MH1, MH2}), whereas it has been mainly studied from the geometric viewpoint in the case of surfaces (\cite{B, TY, Y1, Y2}).

It is worth noticing that if ${X\subset \mathbb{P}^n}$ is a $k$-dimensional integral variety, any projection of $X$ from an $(n-k-1)$-plane gives rise to a dominant rational map ${X\dashrightarrow \mathbb{P}^k}$.
It seems then natural to wonder whether Noether's theorem can be somehow extended in terms of degree of irrationality to higher dimensional subvarieties of $\mathbb{P}^n$.

In this direction, we achieve the following preliminary bound on degree of irrationality of hypersurfaces of general type.
\begin{theorem}\label{theorem CORTINIhypersurfaces}
Let ${X\subset \mathbb{P}^n}$ be a smooth hypersurface of degree ${d\geq n+2}$, with $n\geq 2$. Then
$$
{d-n+1\leq \degree_r(X)\leq d-1}.
$$
\end{theorem}
\noindent We note that ${\degree_r(X)\leq d-1}$ is the obvious bound obtained by projecting $X$ from one of its points.
Moreover, when $n=2$ the statement is just the first assertion of Theorem \ref{theorem NOETHER}.

\smallskip
Apart from plane curves, the first case to investigate is given by smooth surfaces in $\mathbb{P}^3$.
In this setting we prove that Noether's theorem generically holds and we classify the exceptions.
Namely,
\begin{theorem}\label{theorem CORTINIsurfaces}
Let ${X\subset \mathbb{P}^3}$ be a smooth surface of degree ${d\geq 5}$. Then $\degree_r(X)= d-1$ unless one of the following occurs:
\begin{itemize}
  \item[(1)] $X$ contains a twisted cubic;
  \item[(2)] $X$ contains a line $\ell$ and a rational curve $C$ of degree $c$ such that $\ell$ is a ${(c-1)}$-secant line of $C$;
\end{itemize}
in these cases ${\degree_r(X)= d-2}$.\\
In particular, if $X$ is generic then ${\degree_r(X)= d-1}$, and---if in addiction ${d\geq 6}$---any dominant rational map ${X\dashrightarrow \mathbb{P}^{2}}$ of degree $d-1$ is obtained projecting $X$ from one of its points.
\end{theorem}
In particular, the dominant rational maps ${X\dashrightarrow \mathbb{P}^2}$ of degree $d-2$ are defined by using either the family of 2-secant lines of the rational normal cubic, or the family of lines meeting both $\ell$ and $C$ (cf. Examples \ref{example P^3 LINE and CURVE} and \ref{example P^3 TWISTED CUBIC}). Furthermore, the examples we provide assure the sharpness of the bound in Theorem \ref{theorem CORTINIhypersurfaces} when $n=3$.

We recall that having degree ${1\leq d \leq 3}$ is equivalent to rationality.
In order to establish the degree of irrationality of smooth surfaces of $\mathbb{P}^3$, it then remains to understand only the case $d=4$.
If $X$ is a smooth quartic surface, we have that $2\leq\degree_r(X)\leq 3$, and ${\degree_r(X)= d-1=3}$ when $X$ is generic (cf. \cite{Y1}).
Moreover, every smooth projective $K3$ surface containing a hyperelliptic curve of genus ${g\geq 2}$ is birationally equivalent to a branched double covering of $\mathbb{P}^2$ (\cite{D}), and the same happens as $X$ satisfies either condition (1) or (2) above, but it is not clear whether these conditions characterize the case $\degree_r(X)=2$ when $d=4$.

\smallskip
Turning to smooth threefolds in $\mathbb{P}^4$, we prove an analogous assertion and Noether's theorem extends as follows.
\begin{theorem}\label{theorem MAIN}
Let ${X\subset \mathbb{P}^4}$ be a smooth threefold of degree ${d\geq 7}$. Then $\degree_r(X)= d-1$ unless one of the following occurs:
\begin{itemize}
  \item[(1)] $X$ contains a non-degenerate rational scroll $S$ of degree $s$ and an ${(s-1)}$-secant line $\ell$ of $S$;
  \item[(2)] $X$ contains a non-degenerate rational surface $S$ of degree $s$ and a line $\ell\subset S$ such that the intersection of $S$ outside $\ell$ with the general hyperplane $H$ containing $\ell$ is a rational curve $C$ and $\ell$ is an $(s-2)$-secant line of it;
\end{itemize}
in these cases ${\degree_r(X)= d-2}$.\\
In particular ${\degree_r(X)= d-1}$ when $X$ is generic, and---if in addiction ${d\geq 8}$---any dominant rational map ${X\dashrightarrow \mathbb{P}^{3}}$ of degree $d-1$ is obtained projecting $X$ from one of its points.
\end{theorem}

Therefore the latter result provides a characterization of smooth threefolds of degree $d\geq 7$ in $\mathbb{P}^4$ having degree of irrationality $d-2$, and the maps ${X\dashrightarrow \mathbb{P}^3}$ of degree $d-2$ are still given by the families of lines meeting the subvarieties listed above (cf. Examples \ref{example P^4 LINE and SURFACE} and \ref{example P^4 LINE in a SURFACE}).

Dealing with the degree of irrationality of smooth hypersurfaces of $\mathbb{P}^4$ we have that $X$ is rational as $d=1,2$, and it is well known since \cite{CG} that the cubic threefold has ${\degree_r(X)= 2}$.
On the other hand, we have just partial pieces of information when $X\subset \mathbb{P}^4$ has degree ${4\leq d\leq 6}$.
Finally, it would be interesting to understand whether examples of threefold of larger degree satisfying either condition (1) or (2) exist.

\smallskip
In order to prove our results, we follow the approach of \cite{LP} and we use Mumford's technique of induced differentials (cf. \cite[Section 2]{Mu}).
In particular, given a smooth hypersurface $X\subset \mathbb{P}^n$ of degree ${d\geq 2n-1}$, the general fiber of a dominant rational map ${X\dashrightarrow \mathbb{P}^{n-1}}$ of low degree turns out to consist of collinear points, by satisfying a condition of Cayley-Bacharach type.

Then we prove that any dominant rational map ${X\dashrightarrow \mathbb{P}^{n-1}}$ as above induces a first order congruence of lines of $\mathbb{P}^n$.
Moreover, we describe how the degree of the map is related to the number of components of the fundamental locus of the congruence contained in $X$ (see Theorem \ref{theorem MAPS AND CONGRUENCES}).
So we shift our problem to investigate congruences of lines of $\mathbb{P}^n$ of order one.

We recall that a \emph{congruence of $k$-planes} in $\mathbb{P}^n$ is a flat family of $k$-planes parameterized by a $(n-k)$-dimensional subvariety $B\subset \mathbb{G}(k,n)$ of the Grassmannian.
The \emph{order} of the congruence is defined as the number of $k$-planes passing through a general point of $\mathbb{P}^n$, and the \emph{fundamental locus} is the set of points lying on infinitely many lines of the family.
Congruences of linear subspaces of $\mathbb{P}^n$ are interesting geometric objects that have been deeply studied long since (see e.g. \cite{F,Ku,Ma,S}).
Furthermore, there are several recent works---as \cite{A,ABT,D1,R}---on this topic, and others exploiting congruences for different purposes as \cite{AF, CS,DM2}.
We note that the classification of first order congruences of lines in $\mathbb{P}^3$ started by Kummer has been now completed in dependence of the number and the dimension of the components of the fundamental locus  (cf. \cite{A,R}).
On the other hand, Marletta \cite{Ma} started a systematic study of first order congruences of lines of $\mathbb{P}^4$, and the third author corrected and improved his results in a series of papers \cite{D2,D5,D6,DM1}, providing an almost complete classification, which we improve further in the present paper (see Propositions \ref{proposition F THREE SURFACES} and \ref{proposition F TWO SURFACES}).

Dealing with the degree of irrationality of hypersurfaces, we thus exploit the description of first order congruences of lines in $\mathbb{P}^3$ and $\mathbb{P}^4$ to deduce Theorems \ref{theorem CORTINIsurfaces} and \ref{theorem MAIN} by discussing whenever the irreducible components of the fundamental locus may lie on a smooth hypersurface ${X\subset \mathbb{P}^n}$ of general type.

\smallskip
We would like to note that our argument applies also to hypersurfaces of higher dimensional projective spaces and---up to describe the first order congruences of lines in some $\mathbb{P}^n$---one could obtain further extensions of Noether's theorem.

Looking at every known example of first order congruences of lines, we point out that---except for the congruences of the lines passing through a fixed point---any irreducible component of the fundamental locus is covered by rational curves.
Moreover, the same holds in a lot of other cases which are not classified (see e.g. \cite[Theorem 7]{D2}).
Hence it is natural to conjecture that given a first order congruence of lines in $\mathbb{P}^n$ which is not the star of lines through a fixed point, every irreducible component of the fundamental locus contains rational curves.
On the other hand, a generic hypersurface ${X\subset \mathbb{P}^n}$ of degree ${d\geq 2n-1}$ does not contain rational curves (cf. \cite{Cl}).
Then we would have that such a variety $X$ cannot contain any irreducible component of the fundamental locus, and the argument we use would lead to the following.
\begin{conjecture}
The degree of irrationality of a generic hypersurface ${X\subset \mathbb{P}^n}$ of degree ${d\geq 2n-1}$ is ${\degree_r(X)=d-1}$.\\
If in addiction ${d\geq 2n}$, any dominant rational map ${X\dashrightarrow \mathbb{P}^{n}}$ of degree $d-1$ is obtained projecting $X$ from one of its points.
\end{conjecture}

\smallskip
The plan of the paper is the following. In Section \ref{section CORRESPONDENCES} we introduce Mumford's trace map and we present the main results on correspondences with null trace on hypersurfaces.

In Section \ref{section CONGRUENCES} we turn to first order congruences of lines in $\mathbb{P}^n$.
Firstly, we recall some important properties and their classification in $\mathbb{P}^2$ and $\mathbb{P}^3$.
Then we focus on the case of $\mathbb{P}^4$, summarizing the known results about the description of first order congruences and providing a slight improvement on this topic.

Finally, Section \ref{section DEGREE OF IRRATIONALITY} concerns the degree of irrationality of hypersurfaces.
Initially, we prove Theorem \ref{theorem CORTINIhypersurfaces} and we present our result connecting dominant rational maps of hypersurfaces and first order congruences of lines in $\mathbb{P}^n$.
Then we prove Theorems \ref{theorem NOETHER}, \ref{theorem CORTINIsurfaces} and \ref{theorem MAIN}, and we present some remarks and examples on dominant rational maps of low degree.

\begin{notation}
We shall work throughout over the field $\mathbb{C}$ of complex numbers.
By \emph{variety} we mean a complete reduced algebraic variety over $\mathbb{C}$, unless otherwise stated.
When we speak of a \emph{smooth} variety, we always implicitly assume it to be irreducible.

Given a smooth variety $X$, we denote by $\omega_X$ the canonical sheaf of $X$, and $K_X$ denotes any divisor on $X$ such that ${\mathcal{O}_X(K_K)\cong \omega_X}$.

We say that a property holds for a \emph{general} point ${x\in X}$ if it holds on an open non-empty subset of $X$.
Moreover, we say that ${x\in X}$ is a \emph{very general} -~or \emph{generic}~- point if there exists a countable collection of proper subvarieties of $X$ such that $x$ is not contained in the union of those subvarieties.
\end{notation}

\section{Correspondences with null trace on hypersurfaces}\label{section CORRESPONDENCES}

Following \cite[Section 2]{LP} and \cite[Section 4]{B}, we would like to recall the basic properties of Mumford's induced differentials (cf. \cite[Section 2]{Mu}) rephrased in terms of correspondences.

\smallskip
Let $X$ and $Y$ be two projective varieties of dimension $k$, with $X$ smooth and $Y$ integral.
\begin{definition}\label{definition CORRESPONDENCE}
A \emph{correspondence of degree} $m$ \emph{on} $Y\times X$ is a reduced $k$-dimensional variety ${\Gamma\subset Y\times X}$ such that the projections ${\pi_1\colon \Gamma\longrightarrow Y}$, ${\pi_2\colon \Gamma\longrightarrow X}$ are generically finite dominant morphisms and ${deg\,\pi_1=m}$.
Moreover, if ${deg\,\pi_2=m'}$ we say that $\Gamma$ is a $(m,m')$-\emph{correspondence}.\\
Two correspondences $\Gamma\subset Y\times X$ and $\Gamma'\subset Y'\times X$ are said to be \emph{equivalent} if there exists a birational map ${\varphi\colon Y'\dashrightarrow Y}$ such that ${\Gamma'=\left(\varphi\times id_X\right)^{-1}\left(\Gamma\right)}$.
\end{definition}
So, let ${\Gamma\subset Y\times X}$ be a correspondence of degree $m$.
Let us consider the $m$-fold ordinary product ${X^m=X\times\ldots\times X}$ provided with the natural projections ${p_i\colon X^m\longrightarrow X}$, with ${i=1,\ldots,m}$.
Denoting by $S_m$ the $m$-th symmetric group, let ${X^{(m)}=X^m/S_m}$ be the $m$-fold symmetric product of $X$ and let ${\pi\colon X^m\longrightarrow X^{(m)}}$ be the quotient map.
We then define the set ${U:=\{y\in Y_{reg}\,|\,\dim\pi_1^{-1}(y)=0\}}$ and the morphism ${\gamma\colon U \longrightarrow X^{(m)}}$ given by ${\gamma(y):= P_1+\ldots +P_m}$, where ${\pi_1^{-1}(y)=\{(y,P_i)\,|\,i=1,\ldots,m\}}$.

Our first aim is to introduce the trace map of $\gamma$.
Given a holomorphic $k$-form ${\omega\in H^{k,0}(X)}$, let
\begin{equation*}
\omega^{(m)}:=\sum_{i=1}^m p_i^*\omega \in H^{k,0}(X^{m}).
\end{equation*}
As $\omega^{(m)}$ is invariant under the action of $S_m$, the morphism ${\gamma\colon U \longrightarrow X^{(m)}}$ induces canonically a $(k,0)$-form $\omega_{\gamma}$ on $U$ (cf. \cite[Section 2]{Mu}).
Then, we define the Mumford's \emph{trace map} of $\gamma$ as
\begin{displaymath}
\begin{array}{ccccl}
\Tr_{\gamma}\colon & H^{k,0}(X) & \longrightarrow & H^{k,0}(U) \\
 & \omega & \longmapsto & \omega_{\gamma}
\end{array}\, .
\end{displaymath}

We note that $\Tr_{\gamma}$ can be also described locally as follows.
Let ${V:=\{y\in U\,|\pi_1^{-1}(y)\textrm{ has } m \textrm{ distinct points}\}}$ and let
\begin{equation*}
X_0^{(m)}:=\pi\left(X^m-\bigcup_{i,j}\Delta_{i,j}\right),
\end{equation*}
where $\Delta_{i,j}$ is the $(i,j)$-diagonal of $X^{m}$, with ${i,j=1,\ldots, m}$ and ${i\neq j}$.
Moreover, let us consider the map
\begin{displaymath}
\begin{array}{cccc}
\delta_m\colon & H^{k,0}(X) & \longrightarrow & H^{k,0}(X_0^{(m)}) \\
 & \omega & \longmapsto & \pi_*(\omega^{(m)})
\end{array}\, ,
\end{displaymath}
i.e. $\omega^{(m)}$ is thought as a $(k,0)$-form on $X_0^{(m)}$.
Then ${\Image\,\gamma_{|V}\subset X_0^{(m)}}$ and
\begin{equation*}
{\Tr_{\gamma} = \gamma^*_{|V}\circ \delta_m}.
\end{equation*}

\smallskip
We recall the following definition (cf. \cite{GH1}).
\begin{definition}\label{definition CAYLEY-BACHARACH CONDITION}
Let $\mathcal{D}$ be a complete linear system on $X$.
We say that a $0$-cycle ${P_1+\ldots+P_m\in X^{(m)}}$ \emph{satisfies the Cayley-Bacharach condition with respect to} $\mathcal{D}$ if for every ${i=1,\ldots,m}$ and for any effective divisor ${D\in \mathcal{D}}$ passing through ${P_1,\ldots,\widehat{P}_i,\ldots,P_m}$, we have $P_i\in D$ as well.
\end{definition}
We then have the following result showing how the property of having null trace imposes strong conditions on the correspondence ${\Gamma\subset Y\times X}$ (cf. \cite[Proposition 4.2]{B}).
\begin{proposition}\label{proposition LOPEZ PIROLA}
Let $X$ and $Y$ be two projective varieties of dimension $k$, with $X$ smooth and $Y$ integral.
Let $\Gamma$ be a correspondence of degree $m$ on ${Y\times X}$ with null trace.
Let ${y\in Y_{reg}}$ such that ${\dim\pi_1^{-1}(y)=0}$ and let ${\pi_1^{-1}(y)=\{(y,P_i)\in\Gamma\,|\,i=1,\ldots,m\}}$ be its fiber.
Then the $0$-cycle ${P_1+\ldots+P_m}$ satisfies the Cayley-Bacharach condition with respect to the canonical linear series $|\omega_X|$, that is for every ${i=1,\ldots,m}$ and for any effective canonical divisor $K_X$ containing ${P_1,\ldots,\widehat{P}_i,\ldots,P_m}$, we have ${P_i\in K_X}$.
\end{proposition}

\smallskip
We want now to study how the property of having null trace influences the geometry of a correspondence in the case of hypersurfaces of $\mathbb{P}^n$.
We start with the following result extending \cite[Lemma 2.5]{LP} to sets of points in $\mathbb{P}^n$ with $n\geq 2$.
\begin{lemma}\label{lemma LOPEZ PIROLA}
Let ${n\geq 2}$ and let ${Z=\{P_1,\ldots,P_m\}\subset \mathbb{P}^n}$ be a set of distinct points satisfying the Cayley-Bacharach condition with respect to ${\left|O_{\mathbb{P}^n}(r) \right|}$ for some ${r\geq 1}$.
Then ${m\geq r+2}$.
Moreover, if ${m\leq 2r+1}$ then $Z$ lies on a line.
\begin{proof}
We argue by contradiction and we suppose that ${m\leq r+1}$.
By the Cayley-Bacharach condition, any hypersurface of degree $r$ passing through $m-1$ of the $P_i$'s, must contain all of them.
So, let us consider ${m-1}$ hyperplanes ${H_1,\ldots,H_{m-1}\subset \mathbb{P}^n}$ such that ${P_i\in H_i}$ and ${P_m\not\in H_i}$ for any ${i=1,\ldots,m-1}$.
Moreover, let $F_{r-m+1}$ be a hypersurface of degree ${r-m+1}$ not containing ${P_m}$, with the convention ${F_{r-m+1}=\emptyset}$ when ${r=m-1}$.
Therefore ${F_{r-m+1}\cup H_1\cup\ldots\cup H_{m-1}}$ is a hypersurface of degree $r$ containing all but one the $P_i$'s, a contradiction.

Then let us assume ${r+2\leq m\leq 2r+1}$.
If ${r=1}$, then $m=3$ and any hyperplane containing $P_1$ and $P_2$ must contain $P_3$ as well. Thus the $P_i$'s are linearly dependent and hence collinear.

On the other hand, let ${r\geq 2}$ and let $h$ be the maximum number of collinear points of $Z$.
Let ${Z'\subset Z}$ be a subset consisting of $h$ points and let ${\ell\subset \mathbb{P}^n}$ be a line such that ${Z=Z'\amalg Z''}$, with ${Z'\subset \ell}$ and ${\ell\cap Z''=\emptyset}$.

By contradiction, we suppose that ${Z''\neq \emptyset}$. We claim that $Z''$ satisfies the Cayley-Bacharach condition with respect to ${\left|O_{\mathbb{P}^n}(r-1) \right|}$.
Indeed, for any $P\in Z''$ and for any hypersurface $F_{r-1}$ of degree ${r-1}$ containing ${Z''\smallsetminus \{P\}}$, we can choose a hyperplane $H$ such that ${\ell\subset H}$ and $P\not \in H$; since $Z$ satisfies the Cayley-Bacharach condition with respect to ${\left|O_{\mathbb{P}^n}(r) \right|}$, and ${F_{r-1}\cup H}$ is a hypersurface of degree $r$ containing ${Z''\smallsetminus \{P\}}$, we have that $P\in F_{r-1}$ as claimed.

Then the cardinality of $Z''$ is $m-h\geq r+1$ by the first part of the proof.
Furthermore, by induction on $r$ we have that $Z''$ lies on a line $\ell''\subset \mathbb{P}^n$.
As $h$ is the maximum number of collinear points of $Z$, we have that ${m-h\leq h}$.
Therefore ${r+1\leq m-h\leq \frac{m}{2}\leq r+\frac{1}{2}}$, a contradiction.
Thus ${Z''=\emptyset}$ and the whole $Z$ lies on the line $\ell$.
\end{proof}
\end{lemma}

Then the following result descends straightforwardly as a corollary.
\begin{theorem}\label{theorem CORRESPONDENCE ON HYPERSURFACES}
Let ${X\subset \mathbb{P}^n}$ be a smooth hypersurface of degree ${d\geq n+2}$.
For some integral $(n-1)$-dimensional variety $Y$, let ${\Gamma\subset Y\times X}$ be a correspondence of degree $m$ with null trace.
Let ${y\in Y_{reg}}$ such that ${\dim\pi_1^{-1}(y)=0}$ and let ${\pi_1^{-1}(y)=\{(y,P_i)\in\Gamma\,|\,i=1,\ldots,m\}}$ be its fiber.\\
Then $m\geq d-n+1$, and if $m\leq 2d-2n-1$ the $0$-cycle ${P_1+\ldots+P_m}$ consists of collinear points.
\begin{proof}
By Proposition \ref{proposition LOPEZ PIROLA}, the set ${\{P_1,\ldots,P_m\}\subset X}$ satisfies the Cayley-Bacharach condition with respect to the canonical linear series ${\left|\omega_X\right|\cong \left|\mathcal{O}_{\mathbb{P}^n}(d-n-1)\right|}$.
As ${d-n-1\geq 1}$, Lemma \ref{lemma LOPEZ PIROLA} assures that ${m\geq (d-n-1)+2}$ and if in addiction ${m\leq 2(d-n-1)+1}$, the $P_i$'s are collinear.
\end{proof}
\end{theorem}

\begin{remark}
In \cite{B} it has been introduced another attempt at extending the notion of gonality to higher dimensional varieties.
Given a projective variety $X$ of dimension $k$, the \emph{degree of gonality} $\degree_o(X)$ of $X$ is defined as the minimum gonality of curves passing through the generic point of $X$.\\
When ${X\subset \mathbb{P}^3}$ is a smooth surface of degree ${d\geq 5}$, then $\degree_o(X)=d-2$ by \cite[Corollary 1.7]{LP}.
More generally, the degree of gonality of a smooth hypersurface ${X\subset \mathbb{P}^n}$ of degree ${d\geq n+2}$ is ${d-n+1\leq\degree_o(X)\leq d-2}$.
Indeed, if $\degree_o(X)=m$ there exists a correspondence with null trace ${\Gamma\subset Y\times X}$ of degree $m$, where $Y$ is an appropriate ruled $k$-dimensional variety (cf. \cite[Example 4.7]{B}).
On the other hand, let $T_xX$ be the tangent space of $X$ at a generic point ${x\in X}$. Then the plane curve of degree $d$ cut out on $X$ by a general plane of $T_xX$ through $x$ is singular at $x$, and hence its gonality is at most $d-2$.
\end{remark}

To conclude this section we present one of the simplest examples of correspondences with null trace on hypersurfaces, which shall be useful to deal with the study of degree of irrationality.
\begin{example}\label{example CORRESPONDENCE INDUCED BY DOMINANT RATIONAL MAPS}
Let ${X\subset \mathbb{P}^n}$ be a smooth hypersurface and let ${f\colon X\dashrightarrow \mathbb{P}^{n-1}}$ be a dominant rational map of degree $m$.
Then the graph of $f$
\begin{equation}\label{equation CORRESPONDENCE INDUCED BY DOMINANT RATIONAL MAPS}
{\Gamma_f:= \overline{\left\{(y,P)\in \mathbb{P}^{n-1}\times X\left| f(P)=y\right.\right\}}}
\end{equation}
is a $(m,1)$-correspondence on ${\mathbb{P}^{n-1}\times X}$ with null trace.
To see this fact, we note that the fiber ${\pi_1^{-1}(y)}$ over a generic point ${y\in \mathbb{P}^{n-1}}$ consists of $m$ distinct points ${(y,P_1),\ldots,(y,P_m)}$, where ${\left\{P_1,\ldots,P_m\right\}=f^{-1}(y)}$. Therefore $\Gamma_f$ is a reduced variety and $\deg\pi_1=m$.
On the other hand, given a generic point ${P\in X}$, we have ${\pi_2^{-1}(P)=\left\{\left(f(P),P\right)\right\}}$ and hence $\deg\pi_2=1$. Moreover, both the projections are dominant because $f$ is.
To check that the correspondence $\Gamma_f$ has null trace, it suffices to note that the indeterminacy locus of the map ${\gamma\colon \mathbb{P}^{n-1}\dashrightarrow X^{(m)}}$ defined above has codimension greater than one, and $H^{n-1,0}(\mathbb{P}^{n-1})$ is trivial.

Vice versa, any $(m,1)$-correspondence ${\Gamma\subset \mathbb{P}^{n-1}\times X}$ defines a degree $m$ dominant rational map ${f_{\Gamma}\colon X\dashrightarrow \mathbb{P}^{n-1}}$, by sending a generic point ${P\in X}$ to ${\left(\pi_1\circ \pi_2^{-1}\right)(P)\in \mathbb{P}^{n-1}}$.

Moreover, it is immediate to check that the map induced by $\Gamma_f$ is $f$ itself, and the correspondence induced by $f_{\Gamma}$ is $\Gamma$ as well.
\end{example}

\section{First order congruences of lines in $\mathbb{P}^n$}\label{section CONGRUENCES}

In order to deal with first order congruences of lines in $\mathbb{P}^n$, let us firstly recall the definition and some important properties.

\smallskip
Let $\mathbb{G}(1,n)$ be the Grassmannian of lines in $\mathbb{P}^n$ and given a point ${b\in \mathbb{G}(1,n)}$, let us denote by $\ell_b\subset \mathbb{P}^n$ the corresponding line.
\begin{definition}
A \emph{congruence of lines} in $\mathbb{P}^n$ is a flat family ${\Lambda\longrightarrow B}$ obtained as the pullback of the universal family under the desingularization $B\longrightarrow B'$ of an irreducible ${(n-1)}$-dimensional subvariety ${B'\subset \mathbb{G}(1,n)}$; the \emph{order} of the congruence is the number of lines of the family passing through the general point of $\mathbb{P}^n$.
\end{definition}
Given a congruence of lines ${\Lambda\longrightarrow B}$ in $\mathbb{P}^n$, there is a natural identification with the incidence variety ${\Lambda:=\left\{(b,P)\in B\times \mathbb{P}^n\left|P\in \ell_b\right.\right\}}$, hence we shall often refer to $B$ as the congruence of lines itself.
Let
\begin{equation}\label{equation INCIDENCE VARIETY}
\begin{CD}
{\Lambda} @. \subset B\times \mathbb{P}^n @>{\phi}>> \mathbb{P}^n \\
 @VVV \\
 B
\end{CD}
\end{equation}
be the restrictions of the projection maps. If the map $\phi$ is not surjective, the congruence has order zero and $\phi(\Lambda)$ is covered by lines.
Otherwise, the order of the congruence is positive and it equals ${\deg \phi}$.
\begin{definition}
A point ${P\in \mathbb{P}^n}$ is a \emph{fundamental point} of the congruence if ${\dim\phi^{-1}(P)>0}$, and the set $F$ of fundamental points is called \emph{fundamental locus}.
If $R\subset \Lambda$ is the ramification divisor of $\phi$, we say that its schematic image $\Phi:=\phi_*R$ is the \emph{focal locus} of the congruence and the points of $\Phi$ are called \emph{focal points}.
\end{definition}

\begin{remark}\label{remark TANGENT}
The fundamental locus is contained in the focal locus, with ${\dim F\leq n-2}$ and ${\dim\Phi\leq n-1}$. Moreover, if $X\subset \Phi$ is an irreducible component of dimension $n-1$, then any line of the congruence is tangent to $X$ at every intersection point which is focal but not fundamental (cf. \cite[Theorem 1.10]{D1} and \cite[Theorem 2.4]{DM1}).
\end{remark}

\begin{remark}\label{remark FIRST ORDER}
If ${\Lambda\longrightarrow B}$ is a first order congruence, then the fundamental and the focal locus coincide set-theoretically (cf. \cite[Corollary 2.6]{D5}), and both $\Lambda$ and $B$ are rational varieties (see \cite[Theorem 7]{D2}).
\end{remark}

\smallskip
Now we turn to the classification of first order congruences of lines in low dimensional projective spaces.

Clearly, a congruence of lines of $\mathbb{P}^2$ is parameterized by a curve $B$ lying on the dual plane $\mathbb{G}(1,2)\cong\mathbb{P}^{2}$, and the order of the congruence equals the degree of such a curve.
In particular, a curve $B\subset\mathbb{G}(1,2)$ is a first order congruence of lines of $\mathbb{P}^2$ if and only if $B$ parameterizes the star of lines through a point $F\in \mathbb{P}^2$, which is the fundamental locus.

On the other hand, the first order congruences of lines of $\mathbb{P}^3$ are classified by the following result (see e.g. \cite[Theorem 2.1]{A} and \cite[Theorem 0.1]{D4}).
\begin{theorem}\label{theorem CONGRUENCES OF P^3}
Let ${B\subset\mathbb{G}(1,3)}$ be a surface.
Then $B$ is a first order congruence of lines in $\mathbb{P}^3$ if and only if the fundamental locus $F$ is one of the following:
\begin{itemize}
  \item[(a)] a point, where $B$ is the star of lines through $F$;
  \item[(b)] a twisted cubic, where $B$ is the family of bisecant lines of $F$;
  \item[(c)] a non-degenerate reducible curve consisting of a rational curve $C$ of degree $c$ and a ${(c-1)}$-secant line $\ell$, where $B$ is the family of lines meeting both $C$ and $\ell$;
  \item[(d)] a non-reduced line $\ell$ and $B$ is the family of lines $\displaystyle \bigcup_{\pi\in B_\ell}B_{\rho(\pi)}$, where $B_\ell\subset \mathbb{G}(2,3)$ is the pencil of planes containing $\ell$, ${\rho\colon B_\ell\longrightarrow \ell}$ is a non-constant map, and for any plane $\pi\in B_\ell$, $B_{\rho(\pi)}\subset \mathbb{G}(1,3)$ is the star of lines on $\pi$ passing through the point $\rho(\pi)\in \ell$.
  \end{itemize}
\end{theorem}

In $\mathbb{P}^4$ the situation is more involved and the classification is not complete.
In Table \ref{table CONGRUENCES OF P^4} below we list all the possible configurations of the focal locus $F$ of a first order congruence of lines $B$ in $\mathbb{P}^4$, and we summarize the known results in each case.

\begin{remark}\label{remark RATIONAL CURVES}
Given a first order congruence of lines in $\mathbb{P}^4$ which is not the star of lines through a point, the results listed in Table \ref{table CONGRUENCES OF P^4} assure that every irreducible component of the fundamental locus is covered by rational curves.
\end{remark}

\begin{table}[!h]
\noindent
\begin{tabularx}{\columnwidth}{|X|X|X|}
  \hline
  \textbf{Fundamental locus} & \textbf{Congruence} & \textbf{Remarks} \\ \hline
  $F$ is an integral surface & trisecant lines to $F$ & $F$ is not a complete intersection (cf. \cite[Theorem 1.3]{D3}) - classified in \cite[Theorem 1.1]{D5}  \\ \hline
  $F=F_1\cup F_2$, where $F_1$ and $F_2$ are integral surfaces & bisecant lines to $F_1$ intersecting $F_2$ & (see Proposition \ref{proposition F TWO SURFACES} below) \\ \hline
  $F=F_1\cup F_2\cup F_3$, where $F_1$, $F_2$ and $F_3$ are integral surfaces & lines intersecting $F_1$, $F_2$ and $F_3$ & (see Proposition \ref{proposition F THREE SURFACES} below) \\ \hline
  $F$ is a non-reduced irreducible surface & lines intersecting $F$ & $F_{\mathrm{red}}$ is a plane (cf. \cite[Theorem 3.1]{D6}) - classified \\ \cline{2-3}
  & bisecant lines to $F$ & $F_{\mathrm{red}}$ is a cubic scroll (cf. \cite[Theorem 3.2]{D6}) - classified \\ \hline
  $F=F_1\cup F_2$, where $F_1$ is a non-reduced irreducible surface and $F_2$ is an integral surface & lines intersecting both $(F_1)_{\mathrm{red}}$ and $F_2$ & one component is a plane - classified in \cite[Proposition 4.2]{D6} \\ \hline
  $F=F_1\cup C$, where $F_1$ is an integral surface and $C$ is an integral curve (possibly contained in $F_1$) & lines intersecting both $F_1$ and $C$ & classified in \cite[Theorem 1]{D2} \\ \hline
  $F$ is a point & star of lines through $F$ & $F$ can be any point of $\mathbb{P}^4$ - classified \\
  \hline
\end{tabularx}

\medskip
\caption{Description of first order congruence in $\mathbb{P}^4$.}\label{table CONGRUENCES OF P^4}
\end{table}

\smallskip
The following result improves \cite[Theorem 7.2]{DM1} and provides a description of first order congruences of lines in $\mathbb{P}^4$ with fundamental locus consisting of three irreducible components.
\begin{proposition}\label{proposition F THREE SURFACES}
Let $B\subset \mathbb{G}(1,4)$ be a first order congruence of lines in $\mathbb{P}^4$ such that the fundamental locus $F$ consists of three integral surfaces $F_1$, $F_2$ and $F_3$.
Then $B$ parameterizes the family of the lines intersecting all the $F_i$'s, and $F=F_1\cup F_2\cup F_3$ is one of the following:
\begin{itemize}
  \item[(a)] $F_1$, $F_2$ and $F_3$ are three planes intersecting properly;
  \item[(b)] $F_1$, $F_2$ are planes intersecting properly and $F_3$ is a rational surface such that for ${i,j=1,2}$ with ${i\neq j}$, the intersection of $F_3$ outside $F_i$ with the general hyperplane $H$ containing $F_i$ is a rational curve $C_i$ of degree $c_i$ and the line ${\ell= F_j\cap H}$ is ${(c_i-1)}$-secant to $C$;
  \item[(c)] $F_1$ is a plane, $F_2$ is a rational scroll with a unisecant curve given by ${F_1\cap F_2}$, and $F_3$ is a rational surface such that its intersection outside $F_1$ with the general hyperplane $H$ containing $F_1$ is a rational curve $C$ of degree $c$ and the line ${\ell\in F_2\cap H}$ residual to ${F_1\cap F_2}$ is ${(c-1)}$-secant to $C$.
  \end{itemize}
Vice versa, the family of lines intersecting $F_1$, $F_2$ and $F_3$ as above is a first order congruence of lines in $\mathbb{P}^4$.
\begin{proof}
Since the general line of the congruence meet any $F_i$ at only one point, we have that the $F_i$'s are rational surfaces by \cite[Theorem 7]{D2}.
Thanks to \cite[Theorem 7.2]{DM1} we have that at least one irreducible component of $F$ is a plane.
Moreover, if all the components intersect properly, then they are all planes as in (a).

Then let $D_{i+j-2}$ denotes the intersection of the surfaces $F_i$ and $F_j$, and let us assume that at least one of the $D_k$'s is a curve.
Without loss of generality, let $F_1$ be a plane component of $F$.

The rest of the proof is essentially based on the following simple remark.
Given a general point ${P\in \mathbb{P}^4}$, let ${H\cong \mathbb{P}^3}$ be the hyperplane spanned by $P$ and $F_1$.
Since the unique line $\ell$ of the congruence $B$ passing through $P$ must intersect $F_1$, we have that ${\ell\subset H}$, and this fact is true for any general point of $H$.
Therefore the congruence $B$ restricts to a first order congruence of lines $B_{|H}$ of ${H\cong \mathbb{P}^3}$.
Moreover, as any line of $B$ contained in $H$ intersects $F_1$, we have that the fundamental locus of $B_{|H}$ is the reducible curve ${F_{|H}}$, whose components are the closures of the intersections of $H$ with $F_2$ and $F_3$ outside $F_1$.
In particular, both $F_{|H}$ and $B_{|H}$ are described by Theorem \ref{theorem CONGRUENCES OF P^3}.

Now, let us suppose that $F_2$ and $F_3$ are not planes and that they intersect at a curve $D_3$.
Given a general hyperplane $H$ containing $F_1$, we have that any line of the congruence must intersect both the components of $F_{|H}$.
Thus $F_{|H}$ is given by a line $\ell_H$ which is a $(c-1)$-secant of a rational curve $C_H$ of degree $c$.
Without loss of generality, we can assume ${\ell_H\subset F_2\cap H}$ and ${C_H\subset F_3\cap H}$.
Since $F_2$ is not a plane, we have that ${D_1=F_1\cap F_2}$ is a plane curve and ${F_2\cap H=\ell_H\cup D_1}$.
Then we can vary $H$ in the pencil of the hyperplanes of $\mathbb{P}^4$ containing $F_1$, and we deduce that $F_2$ is a rational scroll and $F_3$ is a rational surface as in item (c).

So, let us assume that $F_2$ and $F_3$ are two non-planar surfaces intersecting properly and let us consider a general hyperplane ${H\subset \mathbb{P}^4}$ containing $F_1$.
By arguing as above, the focal locus $F_{|H}$ of the first order congruence of lines $B_{|H}$ of ${H\cong \mathbb{P}^3}$ still consists of a rational curve ${C_H\subset F_3}$ of degree $c$ and of a line ${\ell_H\subset F_2}$ which is $(c-1)$-secant to $C_H$.
Hence $F_2$ is a rational scroll and ${D_1=F_1\cap F_2}$ is one of its unisecant curves.
Since ${D_3=F_2\cap F_3}$ is assumed to be a finite set of points, there are finitely many lines of the ruling of $F_2$ passing through such points.
Then we have that ${c:=\deg C_H = 1}$, i.e. $\ell_H$ and $C_H$ are skew lines.
Indeed, if ${c\geq 2}$, the ${c-1}$ points of ${\ell_H\cap C_H}$ would lie on $D_3$ for the general $H$ in the pencil of hyperplanes containing $F_1$, but the line $\ell_H$ varies as we vary the hyperplane $H$ in the pencil.
Thus we deduce that $F_3$ is a rational scroll as well, where ${D_2=F_1\cap F_3}$ is an unisecant curve on $F_3$, and hence we have case (c) with $c=1$.

Finally, let us suppose that $F_1$ and $F_2$ are planes.
We note that they intersect properly.
Indeed if their intersection were a curve, then they would span a hyperplane and every line meeting both of them would lie in such a hyperplane.
Therefore the congruence $B$ would have order zero, a contradiction.
Then at least one of the loci $D_2$ and $D_3$ is a curve.
Notice that the congruence $B$ induces a first order congruence of lines of the general ${H\cong \mathbb{P}^3}$ containing either $F_1$ or $F_2$.
Thus we can argue as above and we see that $F_3$ is the surface in (b).

To conclude the proof, it is immediate to check that that the family of lines intersecting three surfaces as above is a first order congruence of lines of $\mathbb{P}^4$.
Namely, for the general point $P\in \mathbb{P}^4$, the lines of the family passing through $P$ are all contained in the hyperplane $H\cong \mathbb{P}^3$ containing both $P$ and the plane $F_1$.
For any choice of $F_2$ and $F_3$ between (a), (b) and (c), these surfaces induce a first order congruence on $H$ given by the lines of the family contained in $H$.
Thus we conclude that there exists a unique line of the family through $P$, and we have a first order congruence of lines of $\mathbb{P}^4$.
 \end{proof}
\end{proposition}

We then conclude this section with a slight improvement of \cite[Theorem 7.1]{DM1}, which shall be involved in the proof of Theorem \ref{theorem MAIN}.

\begin{proposition}\label{proposition F TWO SURFACES}
Let $B\subset \mathbb{G}(1,4)$ be a first order congruence of lines in $\mathbb{P}^4$ such that the fundamental locus $F$ consists of two integral surfaces $F_1$ and $F_2$.
Then $B$ parameterizes the family of the bisecant lines to $F_2$ that intersect $F_1$, and
\begin{itemize}
  \item[(a)] if $F_1$ and $F_2$ intersect properly, $F_1$ is a plane and $F_2$ is a rational cubic scroll (possibly a cone);
  \item[(b)] if $F_1\cap F_2$ is a curve and $F_1$ is a plane, then $F_2$ is such that its intersection outside $F_1$ with the general hyperplane $H$ containing $F_1$ is a rational normal cubic curve;
  \item[(c)] if $F_1\cap F_2$ is a curve and $F_1$ is not a plane, $F_2$ is a rational cubic scroll (possibly a cone) and $F_1$ is a rational surface.
  \end{itemize}
Vice versa, if $F_1,F_2\subset \mathbb{P}^4$ are surfaces as in $\mathrm{(a)}$ or $\mathrm{(b)}$, the family of lines intersecting $F_1$ and bisecant to $F_2$ is a first order congruence of lines in $\mathbb{P}^4$.
\begin{proof}
We note that $F_1$ is a rational surface by \cite[Theorem 7]{D2} as the general line of the congruence is unisecant to $F_1$.
Then assertions (a) and (c) descend from \cite[Theorem 7.1]{DM1}.

So, let us assume that $F_1$ is a plane intersecting $F_2$ at a curve $D$.
By arguing as in the proof of the previous proposition, the congruence $B$ induces a first order congruence of lines $B_{|H}$ on the general hyperplane ${H\cong \mathbb{P}^3}$ containing $F_1$.
In particular, the focal locus $F_{|H}$ of $B_{|H}$ is the curve $C_H$ residual to $D$ in ${F_2\cap H}$.

Notice that $C_H$ must be irreducible, because each irreducible component of $C_H$ describes an irreducible component of $F_2$ as we vary $H$ in the pencil of hyperplanes containing $F_1$, and $F_2$ is irreducible.
Moreover, the general line of the congruence meets $C_H$ at two distinct points.
Then $C_H$ is a rational normal cubic curve by Theorem \ref{theorem CONGRUENCES OF P^3} and assertion (b) follows.

Finally, to prove that the configurations of (a) and (b) give rise to a first order congruence of lines of $\mathbb{P}^4$, it suffices to argue as in Proposition \ref{proposition F THREE SURFACES}.
\end{proof}
\end{proposition}

In order to complete the classification of first order congruences of lines of $\mathbb{P}^4$, it then remains to provide an explicit description of the surfaces involved in cases (b) and (c) in the propositions above, but it seems to be a hard task.

\section{Degree of irrationality of hypersurfaces}\label{section DEGREE OF IRRATIONALITY}

In this section we investigate the degree of irrationality of smooth hypersurfaces and we prove the main results stated in the introduction.

To start we prove Theorem \ref{theorem CORTINIhypersurfaces} asserting that the degree of irrationality $\degree_r(X)$ of a smooth hypersurface of degree ${d\geq n+2}$ satisfies the bound
$$
{d-n+1\leq \degree_r(X)\leq d-1}.
$$
\begin{proof}[Proof of Theorem \ref{theorem CORTINIhypersurfaces}]
Let ${f\colon X\dashrightarrow \mathbb{P}^{n-1}}$ be a dominant rational map of degree $\degree_r(X)$.
Following Example \ref{example CORRESPONDENCE INDUCED BY DOMINANT RATIONAL MAPS}, the Zariski closure of the graph of $f$ is a correspondence ${\Gamma_f:= \overline{\left\{(y,P)\in \mathbb{P}^{n-1}\times X\left| f(P)=y\right.\right\}}}$ with null trace of degree $\degree_r(X)$.
Therefore ${\degree_r(X)\geq d-n+1}$ by Theorem \ref{theorem CORRESPONDENCE ON HYPERSURFACES}.
On the other hand, the projection of $X$ from a point $x\in X$ is a dominant rational map ${X\dashrightarrow \mathbb{P}^{n-1}}$ of degree $d-1$, and hence ${\degree_r(X)\leq d-1}$ as wanted.
\end{proof}

\smallskip
In order to deal with low dimensional projective spaces, we firstly prove some results connecting dominant rational maps of hypersurfaces and first order congruences of lines.

In this direction, we start by constructing a first order congruence of $\mathbb{P}^n$ by means of a dominant rational map ${X\dashrightarrow \mathbb{P}^{n-1}}$ of low degree.

\begin{lemma}
Let ${X\subset \mathbb{P}^n}$ be a smooth hypersurface of degree ${d\geq 2n+1-k}$ for some ${0\leq k\leq 2}$, and let ${f\colon X\dashrightarrow \mathbb{P}^{n-1}}$ be a dominant rational map of degree ${m\leq d-k}$.
Then $f$ induces a first order congruence of lines ${\Lambda_{f}\longrightarrow B_{f}}$ of $\mathbb{P}^n$.
\begin{proof}
Let ${f\colon X\dashrightarrow \mathbb{P}^{n-1}}$ be a dominant rational map of degree ${m\leq d-k}$ and let us consider the ${(m,1)}$-correspondence with null trace ${\Gamma_f:= \overline{\left\{(y,P)\in \mathbb{P}^{n-1}\times X\left| f(P)=y\right.\right\}}}$ defined as the graph of $f$ (cf. Example \ref{example CORRESPONDENCE INDUCED BY DOMINANT RATIONAL MAPS}).
Let ${\pi_1\colon \Gamma_f\longrightarrow \mathbb{P}^{n-1}}$ be the first projection map and for a generic ${y\in \mathbb{P}^{n-1}}$, let ${\pi_1^{-1}(y)=\left\{(y,P_i)|i=1,\ldots,m\right\}}$ be its fiber.
Since ${d\geq 2n+1-k}$ and ${m\leq d-k}$, we have ${m\leq 2d-2n-1}$ and hence the set of points ${f^{-1}(y)=\{P_1,\ldots,P_m\}}$ lies on a line ${\ell_y\in \mathbb{P}^n}$ by Theorem \ref{theorem CORRESPONDENCE ON HYPERSURFACES}.

We note that the smooth hypersurface $X$ is not covered by lines, otherwise it would be uniruled, which is not the case because $H^{n-1,0}(X)\cong H^0\left(\mathbb{P}^n,\mathcal{O}_{\mathbb{P}^n}(d-n-1)\right)\neq \{0\}$.
Then the line $\ell_y$ is not contained in $X$.
Furthermore ${m\geq d-n+1}$ by Theorem \ref{theorem CORTINIhypersurfaces}.
Thus $m> \frac{d}{2}$ as ${d\geq 2n+1-k}$ and ${0\leq k\leq 2}$.
It follows that for a generic point ${y\in \mathbb{P}^{n-1}}$, does not exist another point ${y'\in \mathbb{P}^{n-1}}$ such that $f^{-1}(y')\subset \ell_y$.
By sending $y\in \mathbb{P}^{n-1}$ to the point of $\mathbb{G}(1,n)$ parameterizing ${\ell_y\subset \mathbb{P}^n}$, we can then define a rational map ${\mathbb{P}^{n-1}\dashrightarrow \mathbb{G}(1,n)}$ of degree one, whose image is a ${(n-1)}$-dimensional subvariety $B'$ of the Grassmannian.

Then we consider a desingularization ${B_{f}\longrightarrow B'}$ and we have that
\begin{equation}\label{equation INDUCED CONGRUENCE}
\Lambda_{f}:=\left\{(b,P)\in B_f\times \mathbb{P}^n | P\in \ell_b\right\}\longrightarrow B_f
\end{equation}
is a congruence of lines in $\mathbb{P}^n$ induced by the map $f$.

We note that $B_f$ has order $a\geq 1$.
Indeed, if it were ${a=0}$, the general line of the congruence intersecting $X$ would lie on it and $X$ would be covered by lines.

In order to prove that $a=1$, we firstly prove that there is only one line of the congruence passing through the general point $P\in X$.
When $k=0$, let $y=f(P)$ and let $\ell_y$ the unique line containing the fiber $f^{-1}(y)$.
Then $\deg f=d$ and hence $f^{-1}(y)=X\cap \ell_y$.
Thus $\ell_y$ must be the unique line of the congruence through $P$, otherwise $P$ would belong to different fibers.

If instead $k>0$, let us consider the variety residual to $\Gamma_f$ defined as
\begin{equation}\label{equation RESIDUAL CONGRUENCE}
\Gamma'_f:=\overline{\left\{(b,P)\in B_f\times X \left| \begin{array}{l}P\in \ell_b \textrm{ and } P\not\in f^{-1}(y),\textrm{ where }\\  y\in \mathbb{P}^{n-1}\textrm{ and }f^{-1}(y)\subset \ell_b\end{array}\right.\right\}}
\end{equation}
together with the natural projections ${\pi_1'\colon \Gamma'_f\longrightarrow B_f}$ and ${\pi_2'\colon \Gamma'_f\longrightarrow X}$.
Since the general line $\ell_b$ meets $X$ at $d$ points and ${\deg f=m}$, we have that $\pi_1'$ is a dominant morphism of degree ${d-m}$.

We claim that $\pi_2'$ is not dominant.
To see this fact we firstly suppose that the points of the general fiber $(\pi_1')^{-1}(y)$ are distinct, so that $\Gamma'_f$ is a reduced variety.
If $\pi_2'$ were a dominant map, we would have that $\Gamma'_f$ would be a correspondence on ${B_f\times X}$.
Moreover, $\Gamma'_f$ would have null trace as $B_f$ is rational and ${H^{n-1,0}\left(B_f\right)=\left\{0\right\}}$.
Then ${\deg \pi_1'\geq d-n+1}$ by Theorem \ref{theorem CORRESPONDENCE ON HYPERSURFACES}, but this would lead to a contradiction.
Indeed $d\geq 2n+1-k$ and $m\geq d-n-1$, hence $\deg \pi_1'=d-m\leq n-1\leq d-n-2+k\leq d-n$ when ${1\leq k\leq 2}$.

On the other hand, if the points of the general fiber $(\pi_1')^{-1}(y)$ are not distinct, the irreducible components of $\Gamma'_f$ are not all reduced.
Then we can consider the variety $(\Gamma'_f)_{\mathrm{red}}$ obtained by taking the reduction of any component of $\Gamma'_f$.
By assuming $\pi_2'$ to be dominant, $(\Gamma'_f)_{\mathrm{red}}$ would be a correspondence with null trace of degree smaller than ${d-m}$, and---by arguing as above---we would have a contradiction.

Therefore $\pi_2'$ is not dominant and the image ${\pi_2'(\Gamma'_f)\subset X}$ is a finite union of proper subvarieties.
Hence there is only one line passing through the general point of $X$.

So, we assume by contradiction that the order of $B_f$ is ${a>1}$.
As there is only one line of the congruence passing through the general point of $X$, we deduce that the generically finite map ${\phi\colon \Lambda_{f}\longrightarrow \mathbb{P}^{n}}$ defined in (\ref{equation INCIDENCE VARIETY}) is totally ramified over any point of $X$.
Thus $X$ is a $(n-1)$-dimensional irreducible component of the focal locus $\Phi$ of the congruence $B_f$.
Then the general line $\ell_b$ of the congruence is tangent to $X$ at every intersection point which is focal but not fundamental (cf. Remark \ref{remark TANGENT}).
Therefore $\ell_b\cap X$ consists of at most $\frac{d}{2}$ distinct points.
Hence we have a contradiction because $\ell_b\cap X$ contains a fiber of $f$ and $\deg f>\frac{d}{2}$.

We thus conclude that ${\Lambda_{f}\longrightarrow B_{f}}$ is a first order congruence of lines in $\mathbb{P}^n$.
\end{proof}
\end{lemma}

\smallskip
Vice versa, we want now to show that first order congruences of lines induce dominant rational map of hypersurfaces.
Given a nonsingular hypersurface ${X\subset \mathbb{P}^n}$ and a first order congruence of lines ${\Lambda\longrightarrow B}$ in $\mathbb{P}^n$ with fundamental locus $F$, we denote by $F_{B|X}$ the union of the components of $F$ entirely contained in the hypersurface $X$, and by $\delta_{B|X}$ we mean the number of intersection points---counted with multiplicity---of a general line of the congruence with $X$ at ${F_{B|X}}$.

Moreover, we say that two dominant rational maps ${f\colon X\dashrightarrow \mathbb{P}^{n-1}}$ and ${f'\colon X\dashrightarrow \mathbb{P}^{n-1}}$ are \emph{equivalent} if they differ for a Cremona transformation of $\mathbb{P}^{n-1}$, i.e. if there exists a birational map ${\varphi\colon \mathbb{P}^{n-1}\dashrightarrow \mathbb{P}^{n-1}}$ such that ${f'=\varphi\circ f}$.

With this notation we have the following.
\begin{lemma}
Let ${X\subset \mathbb{P}^n}$ be a smooth hypersurface of degree ${d\geq n+2}$ and let ${\Lambda\longrightarrow B}$ be a first order congruence of lines in $\mathbb{P}^n$.
Then $B$ induces a dominant rational map $f_{B}\colon X\dashrightarrow \mathbb{P}^{n-1}$ of degree ${m=d-\delta_{B|X}}$, which is unique up to equivalence.
\begin{proof}
Let us consider the variety
\begin{equation}\label{equation CORRESPONDENCE INDUCED BY A CONGRUENCE}
\Gamma_{B}:=\overline{\{(b,P)\in B\times X|P\in \ell_b \textrm{ and } P\not\in F_{B|X}\}}\subset B\times X
\end{equation}
provided with the projection maps ${\pi_1\colon \Gamma\longrightarrow B}$ and ${\pi_2\colon \Gamma\longrightarrow X}$.
As the congruence has order one, there exists a unique line of the family passing through the general point $P\in X$, and hence ${\deg\pi_2=1}$.
Moreover, such a line does not lie on $X$, otherwise the hypersurface would be uniruled, but this is impossible as $X$ is of general type.
Therefore the general line $\ell_b$ of the congruence ${\Lambda\longrightarrow B}$ meets $X$ at ${d-\delta_{B|X}}$ distinct points outside $F_{B|X}$, that is ${m:=\deg\pi_1=d-\delta_{B|X}}$.
Thus ${\Gamma_{B}}$ is a ${(m,1)}$-correspondence on ${B\times X}$, and hence it induces a dominant rational map ${\gamma_{B}\colon X\dashrightarrow B}$ of degree ${m=d-\delta_{B|X}}$ as in Example \ref{example CORRESPONDENCE INDUCED BY DOMINANT RATIONAL MAPS}.

Since $B$ is birational to $\mathbb{P}^{n-1}$ (cf. Remark \ref{remark FIRST ORDER}), we may then define a dominant rational map
\begin{equation}\label{equation INDUCED MAP}
f_B\colon\xymatrix{ X \ar@{-->}[rr] \ar[dr]_{\gamma_B} & & \mathbb{P}^{n-1} \ar@{<--}[dl] \\ & B & \\ }
\end{equation}
of degree ${m=d-\delta_{B|X}}$, which is unique up to equivalence.
\end{proof}
\end{lemma}

\smallskip
We then have the following.

\begin{theorem}\label{theorem MAPS AND CONGRUENCES}
Let ${X\subset \mathbb{P}^n}$ be a smooth hypersurface of degree ${d\geq 2n+1-k}$ for some ${0\leq k\leq 2}$. Then the following hold:
\begin{itemize}
  \item[(i)] any dominant rational map ${f\colon X\dashrightarrow \mathbb{P}^{n-1}}$ of degree ${m\leq d-k}$ induces a first order congruence of lines ${\Lambda_{f}\longrightarrow B_{f}}$ of $\mathbb{P}^n$ as in (\ref{equation INDUCED CONGRUENCE});
  \item[(ii)] any first order congruence of lines ${\Lambda\longrightarrow B}$ of $\mathbb{P}^n$ induces a dominant rational map $f_{B}\colon X\dashrightarrow \mathbb{P}^{n-1}$ of degree ${m=d-\delta_{B|X}}$ as in (\ref{equation INDUCED MAP}).
\end{itemize}
In particular, the dominant rational map induced by the congruence ${B_{f}}$ is equivalent to $f$ and---vice versa---the congruence induced by $f_B$ is $B$.
\begin{proof}
Thanks to the previous lemmas, we only need to prove the last part of the statement.
Let ${f\colon X\dashrightarrow \mathbb{P}^{n-1}}$ be a dominant rational map as above and let us consider the variety ${\Gamma'_f\subset B_f\times X}$ in (\ref{equation RESIDUAL CONGRUENCE}).
As the second projection map ${\pi_2'\colon \Gamma'_f\longrightarrow X}$ is not dominant, we have that any irreducible component of $\Gamma'_f$ maps on a proper subvariety of $X$.
Moreover, any such a subvariety is met by all the lines parameterized by $B_f$, and hence it is a component of the fundamental locus $F$ of $B_f$.
Vice versa, the pullback via ${\pi_2'}$ of any component of $F$ contained in $X$ is a component of $\Gamma'_f$.
Thus ${\pi_2'(\Gamma'_f)}$ coincides with the union $F_{B_f|X}$ of the components of the fundamental locus entirely contained in $X$.
Therefore ${\Gamma'_f=\overline{\{(b,P)\in B_f\times X| P\in \ell_b \textrm{ and } P\in F_{B_f|X}\}}}$ is residual to the correspondence $\Gamma_{B_f}$ in (\ref{equation CORRESPONDENCE INDUCED BY A CONGRUENCE}).
Hence $\Gamma_f$ and $\Gamma_{B_f}$ are equivalent correspondences inducing equivalent dominant rational maps from $X$ to $\mathbb{P}^{n-1}$.

Analogously, given a first order congruence $B\subset \mathbb{G}(1,n)$, we can consider the correspondence ${\Gamma_{f_B}}$ and it is immediate to see that the map $f_B$---defined by the correspondence $\Gamma_B$---induces the congruence $B$.
\end{proof}
\end{theorem}

\begin{remark}
Let ${X\subset \mathbb{P}^n}$ be a smooth hypersurface of degree ${d\geq 2n+1}$.
As a corollary of the previous result, we have that there is a one-to-one correspondence between the set of first order congruences of lines in $\mathbb{P}^n$ and the set of equivalence classes of dominant rational maps ${X\dashrightarrow \mathbb{P}^{n-1}}$ of degree at most $d$.
\end{remark}

\begin{remark}\label{remark PROOF OF NOETHER}
We note that by setting $n=2$ and $k=1$ in the assertion of Theorem \ref{theorem MAPS AND CONGRUENCES}, it is immediate to deduce Theorem \ref{theorem NOETHER} on the gonality of plane curves of degree at least 4.
On one hand, given a smooth curve ${X\subset \mathbb{P}^2}$ of degree ${d\geq 4}$, we have that $\gon(X)=\degree_r(X)=d-1$ by Theorem \ref{theorem CORTINIhypersurfaces}.
On the other, any map ${f\colon X\dashrightarrow \mathbb{P}^1}$ of degree $d-1$ is induced by a first order congruence of lines $B_f$ in $\mathbb{P}^2$.
Since $B_f$ is a star of lines with fundamental locus consisting of a point ${Q\in \mathbb{P}^2}$ (cf. Section \ref{section CONGRUENCES}) and ${\deg f=d-1=d-\delta_{B|X}}$, we conclude that ${Q\in X}$ and ${f\colon X\dashrightarrow \mathbb{P}^1}$ is the projection from $Q$.
\end{remark}

\smallskip
Now, in order to prove Theorem \ref{theorem CORTINIsurfaces}, we need the following preliminary lemma.
\begin{lemma}\label{lemma TANGENT LINES}
Let ${X\subset \mathbb{P}^3}$ be a smooth surface of degree ${d\geq 5}$ and let ${\ell\subset X}$ be a line.
Then the congruence $B\subset \mathbb{G}(1,3)$ of the tangent lines of $X$ along $\ell$ has not order one.
\begin{proof}
Aiming for a contradiction, let us assume that $B$ is a first order congruence of lines of $\mathbb{P}^3$.
Therefore $B$ is the congruence described in (d) of Theorem \ref{theorem CONGRUENCES OF P^3}.

Now, let $B_\ell$ be the pencil of planes containing $\ell$ and let ${H\subset \mathbb{P}^3}$ be a general element.
Then ${H\cap X=\ell\cup C_H}$, where $C_H$ is a curve of degree $d-1$.
As $C_H$ is a general element of the pencil ${|\mathcal{O}_X(H-\ell)|}$ on $X$, $C_H$ is smooth outside the base locus of the linear system by Bertini's Theorem.
In particular, ${C_H}$ is smooth outside $\ell$ and the plane $H$ is tangent to $X$ at any point of ${D:=C_H\cap \ell}$.
Hence for any ${Q\in D}$, every line $\ell_Q\subset H$ through $Q$ belongs to the first order congruence $B$.
Thus the support of $D$ must consist of only one point ${Q\in \ell}$, otherwise we would have two distinct lines $\ell_Q$ and $\ell_{Q'}$ passing through the general point of $H$.

We note that ${D\in \Div^{d-1}(\ell)}$ is an general effective divisor varying in the pencil cut out on $\ell$ by the curves $C_H$, and ${D=(d-1)Q}$.
Therefore $D$ is a singular at $Q$, and hence $Q$ is a base point of the latter pencil by Bertini's Theorem.
Then for any plane $H$ containing $\ell$, the curve $C_H$ passes through ${Q\in\ell}$.
Thus any plane ${H\in B_\ell}$ is tangent to $X$ at $Q$, i.e. $Q$ is a singular point of $X$, a contradiction.
\end{proof}
\end{lemma}

\smallskip
So, let us prove Theorem \ref{theorem CORTINIsurfaces}.

\begin{proof}[Proof of Theorem \ref{theorem CORTINIsurfaces}]

Let $X\subset \mathbb{P}^3$ be a smooth surface of degree $d\geq 5$ and let ${f\colon X\dashrightarrow \mathbb{P}^2}$ be a dominant rational map of degree $\degree_r(X)$.
Thanks to Theorem \ref{theorem CORTINIhypersurfaces} we have $d-2\leq \degree_r(X)\leq d-1$.

If ${\degree_r(X)=d-2}$, Theorem \ref{theorem MAPS AND CONGRUENCES} assures that $f$ is equivalent to a map induced by the first order congruence of lines $B_f$ of $\mathbb{P}^3$.
Furthermore, the general line of the congruence must intersect twice the locus $F_{B_f|X}$ as ${\degree_r(X)=d-\delta_{B_f|X}}$.
Clearly, the multiplicity of the intersection of a general line of $B_f$ with $X$ at a general point of a reduced component of $F_{B_f|X}$ is one (otherwise the component would be non-reduced).

By Theorem \ref{theorem CONGRUENCES OF P^3}, if the fundamental locus $F$ is such that the general line of $B_f$ meets $F$ at two distinct points, then $F$ is either a twisted cubic or a reducible curve consisting of a non-degenerate rational curve of degree $c$ and a ${(c-1)}$-secant line.
In both these cases, $X$ must contain the whole $F$ as ${\delta_{B_f|X}=2}$, and we obtain cases (1) and (2) of the statement.

Moreover, the congruence $B_f$ cannot be described by (d) of Theorem \ref{theorem CONGRUENCES OF P^3}.
Indeed, we would have that the general line meets the fundamental locus $\ell=F_{B_f|X}$ at one point with multiplicity $2=\delta_{B_f|X}$.
Thus $B_f$ would be the congruence of tangent lines to $X$ along $\ell$, but this is impossible by Lemma \ref{lemma TANGENT LINES}.

Let now $X$ be a generic surface of degree $d\geq 5$.
We recall that it does not contain rational curves (see \cite{X}), and hence $\degree_r(X)=d-1$ by the first part of the proof.

Finally, let us assume that $X$ is a generic surface of degree ${d\geq 6}$.
Then by Theorem \ref{theorem MAPS AND CONGRUENCES} we have that the dominant rational map ${f\colon X\dashrightarrow \mathbb{P}^2}$ induces a first order congruence $B_f$ and ${\deg f=d-1=d-\delta_{B_f|X}}$.
Thus $X$ must contain a component of the fundamental locus of $B_f$.
As $X$ does not contain rational curves, we deduce from Theorem \ref{theorem CONGRUENCES OF P^3} that such a component must be a point $Q\in X$ and the dominant rational map $f$ is the projection from $Q$.
\end{proof}

\begin{remark}
We note that the original proof of Theorem \ref{theorem CORTINIsurfaces} included in \cite{Co} differs from the one above.
Indeed the second author started from the classification in \cite{LP} of correspondences with null trace on smooth surfaces in $\mathbb{P}^3$, that has been obtained using the techniques of Section \ref{section CORRESPONDENCES} as well.
Then she investigated---with a lot of effort---whenever these correspondences have bidegree $(d-2,1)$ on smooth surfaces of degree $d\geq 5$.
In our proof, we approach similarly with the use of correspondences and Mumford's induced differentials, but we then shift the problem to congruences of lines of $\mathbb{P}^3$---which are classified---so that the result is almost straightforward.
\end{remark}

\begin{example}\label{example P^3 LINE and CURVE}
According to (2) of Theorem \ref{theorem CORTINIsurfaces}, let ${X\subset \mathbb{P}^3}$ be a smooth surface containing both a rational curve $C$ of degree $c$ and a $(c-1)$-secant line $\ell$ of $C$.
Then the map ${f\colon X\dashrightarrow \mathbb{P}^2}$ of degree $d-2$ can be explicitly viewed as follows.\\
Let $P\in X$ be a general point and let ${H\subset \mathbb{P}^3}$ be the plane through $P$ containing $\ell$.
Notice that $H$ cuts out on $X$ the line $\ell$ and a curve $D_H$ of degree $d-1$, with ${P\in D_H}$.
Moreover, the plane $H$ meets the curve $C$ at $c$ points and there exists a unique point ${P_H\in H\cap C}$ not lying on $\ell$, because $\ell$ is a $(c-1)$-secant line of $C$.
Since ${P_H\in D_H}$, the projection ${f_H\colon D_H\longrightarrow \mathbb{P}^1}$ from the point $P_H$ has degree $d-2$.
Thus it is well defined the dominant rational map ${X\dashrightarrow B_\ell\times \mathbb{P}^1}$ sending a general point ${P\in X}$ to the pair ${\left(h,f_H(P)\right)\in B_\ell\times \mathbb{P}^1}$, where $h$ is the point parameterizing $H$ in the pencil ${B_\ell\subset \mathbb{G}(2,3)}$ of planes containing $\ell$.
Clearly, such a map has degree $d-2$ because any $f_H$ has.
Finally, being ${B_\ell\times \mathbb{P}^1}$ and $\mathbb{P}^2$ birational surfaces, we obtain the map $f$.\\
We note further that if $c=1$, then $\ell$ and $C$ are skew lines contained in $X$, and it is possible to provide examples of this setting for any degree $d\geq 5$.
For instance, we can consider the smooth surface ${X\subset \mathbb{P}^3}$ defined by the equation ${x_0^d-x_0x_1^{d-1}+x_2^d-x_2x_3^{d-1}=0}$, which contains the skew lines ${x_0=x_2=0}$ and ${x_0-x_1=x_2-x_3=0}$.
\end{example}

\begin{example}\label{example P^3 TWISTED CUBIC}
When ${X\subset \mathbb{P}^3}$ is smooth surface containing a rational normal cubic $C$ as in (1) of Theorem \ref{theorem CORTINIsurfaces}, the map ${f\colon X\dashrightarrow \mathbb{P}^2}$ of degree $d-2$ is the following.\\
For the general point $P\in X$, there exists a unique line $\ell_P$ through $P$ intersecting twice the curve $C$.
As $C\cong \mathbb{P}^1$, its second symmetric product is ${C^{(2)}\cong \mathbb{P}^2}$, and $f$ is the map sending a point $P\in X$ to the point $Q_1+Q_2\in C^{(2)}$, where $Q_1$ and $Q_2$ are the intersection points of $C$ with the bisecant line $\ell_P$.\\
We also point out that for any sufficiently large $d$, there exists a smooth surface $X\subset \mathbb{P}^3$ of degree $d$ containing a given twisted cubic $C$ (see e.g. \cite[p. 355]{H3}).
\end{example}

\begin{remark}
We would like to recall that any dominant rational map ${C'\dashrightarrow C}$ between smooth curves
leads to the inequality ${\degree_r(C')\geq \degree_r(C)}$, but this fact is no longer true for higher dimensional varieties.
There are indeed examples of non-rational threefolds that are unirational (cf. \cite{CG, IM}), and counterexamples in the case of surfaces (cf. \cite{Y2}).
In this direction, we have also that the general smooth surface ${S\subset \mathbb{P}^3}$ of degree $d\geq 5$ containing a fixed line has $\degree_r(S)=d-1$, and it is dominated by a surface $S'$ having $\degree_r(S')=d-2$.\\
To see this fact, let ${\ell\subset \mathbb{P}^3}$ be a line and let ${S\subset \mathbb{P}^3}$ be general between smooth surfaces of degree $d$ containing $\ell$.
By \cite[Theorem II.3.1]{L} and Theorem \ref{theorem CORTINIsurfaces} it is easy to check that $\degree_r(S)=d-1$.
Then let $B_\ell$ be the pencil of planes containing $\ell$.
Let ${g\colon \ell\longrightarrow B_\ell}$ and ${h\colon S\longrightarrow B_\ell}$ be the morphisms given by ${g(Q)=T_Q(S)}$, ${h(P)=<\ell,P>}$ if ${P\in S-\ell}$, and ${h(P)=T_P(S)}$ otherwise.
We can then consider the fibred product $S':={S\times_{B_\ell}\ell\cong S\times_{\mathbb{P}^1}\mathbb{P}^1}$ together with the dominant map ${p\colon S'\longrightarrow S}$.
Notice that ${(P,Q)\in S'}$ if and only if ${h(P)\cap S=\ell\cup C_P}$ where $C_P$ is a curve of degree $d-1$ passing through $P$ and $Q$.
Thus we can define a dominant rational map ${S'\dashrightarrow \mathbb{P}^1\times \ell}$ of degree $d-2$ sending ${(P,Q)\in S'}$ to the point ${(f_Q(P),Q)}$, where $f_Q\colon C_P\longrightarrow \mathbb{P}^1$ is the projection from $Q$ of the plane curve $C_P$.
Hence ${\degree_r(S')\leq d-2}$.
On the other hand, any dominant rational map ${f\colon S'\dashrightarrow \mathbb{P}^2}$ of degree ${\deg f=\degree_r(S')}$ induces a correspondence with null trace ${\Gamma=\{(y,P)\in \mathbb{P}^2\times S| f^{-1}(y)\cap p^{-1}(P)\neq \emptyset\}}$ of degree equal to $\deg f$.
Thus ${\deg f\geq d-2}$ by Theorem \ref{theorem CORRESPONDENCE ON HYPERSURFACES} and hence ${\degree_r(S')=d-2}$.
\end{remark}

\smallskip
Now we turn to hypersurfaces in $\mathbb{P}^4$ and we similarly prove Theorem \ref{theorem MAIN}.
We start with a simple preliminary result.
\begin{lemma}\label{lemma SCROLL}
Let $X\subset \mathbb{P}^4$ be a smooth hypersurface of degree $d\geq 7$ and let ${S\subset X}$ be a two-dimensional rational scroll possessing a unisecant plane curve $C$.
Then $S$ is degenerate.
\begin{proof}
Let $s:=\deg S$ and let $\pi$ be the plane containing $C$.
As $C$ is unisecant, we have ${c:=\deg C=s-1}$.
We recall that any codimension one irreducible subvariety of $X$ is a complete intersection by Lefschetz Theorem (see e.g. \cite[Chapter IV]{H1}).
Therefore $s=kd$ for some integer $k\geq 1$.
Notice that $\pi\cap X$ is a curve of degree $d$ containing $C$.
Thus $d\geq c=kd-1$ and hence $k=1$, i.e. $S$ is degenerate.
\end{proof}
\end{lemma}

\smallskip
\begin{proof}[Proof of Theorem \ref{theorem MAIN}]
Let $X\subset \mathbb{P}^4$ be a smooth threefold of degree $d\geq 7$ and let ${f\colon X\dashrightarrow \mathbb{P}^3}$ be a dominant rational map of degree ${\deg f=\degree_r(X)}$.
Thanks to Theorem \ref{theorem CORTINIhypersurfaces} we have ${d-3\leq \deg f\leq d-1}$.

Let us assume that $\deg f\leq d-2$ and let $B_f$ be the first order congruence of lines of Theorem \ref{theorem MAPS AND CONGRUENCES} with fundamental locus $F$.
Since ${\deg f=d-\delta_{B_f|X}}$ we have to prove that ${\delta_{B_f|X}\neq 3}$ and to investigate when ${\delta_{B_f|X}=2}$.
The main point is to check which components of $F$ can be subvarieties of $X$.
To this aim, we recall that any codimension one irreducible subvariety of $X$ is a complete intersection by Lefschetz Theorem, and in particular, it must have degree at least $d$.

We note that the general line of a first order congruence intersects the focal locus at three points---counted with multiplicity---if and only if all the components of $F$ are surfaces.
Then ${\delta_{B_f|X}= 3}$ if and only if $X$ contains the whole fundamental locus, that is ${F=F_{B_f|X}}$.
Looking at the summarizing table in Section \ref{section CONGRUENCES}, it is immediate to check that this situation cannot occur: if $F$ is an integral surface, it is not complete intersection by \cite[Theorem 1.3]{D3}, and in the remaining cases one component of $F$ is either a plane or a cubic surface.

Now, let us study the cases leading to ${\delta_{B_f|X}=2}$.
When $F$ is an irreducible---possibly non-reduced---surface, we argue as above and we exclude these cases because $F$ is not a complete intersection of $X$.

Suppose that $F=F_1\cup F_2$, where $F_1$ and $F_2$ are integral surfaces, and $B_f$ is the family of bisecant lines to $F_2$ meeting also $F_1$.
In order to have ${\delta_{B_f|X}= 2}$, it must be ${F_2\subset X}$.
Since $F_2$ cannot be a cubic surface, Proposition \ref{proposition F TWO SURFACES} guarantees that $F_1$ is a plane and $F_2$ a non-degenerate surface such that its intersection outside $F_1$ with the general hyperplane $H\supset F_1$ is a twisted cubic $C_H$.
Arguing as in the Lemma \ref{lemma SCROLL}, if ${F_2\subset X}$, the curve $F_1\cap F_2$ would have degree $\deg F_2-3$, where $\deg F_2=kd$ for some positive integer $k$.
Moreover, $F_1\cap X$ would be a curve of degree $d$ containing $F_1\cap F_2$.
Thus ${d\geq kd-3}$ which leads to $k=1$.
Then $F_2$ would be degenerate, a contradiction.
Hence $X$ does not contain such an $F_2$.

Assume $F=F_1\cup F_2\cup F_3$, where all the components are integral surfaces and $B_f$ is the family of lines meeting all of them.
Thus two components of $F$ must be contained in $X$, but this is impossible by Proposition \ref{proposition F THREE SURFACES} and Lemma \ref{lemma SCROLL}.

Consider the case $F=F_1\cup F_2$, where $F_1$ is non-reduced irreducible surface and $F_2$ is an integral one.
Then $(F_1)_{\mathrm{red}}$ must be contained in $X$, but this is impossible.
Indeed the classification of \cite[Proposition 4.2]{D6} assures that $(F_1)_{\mathrm{red}}$ is either a plane or a non-degenerate rational scroll provided with a unisecant plane curve.

Since the locus $F_{B_f|X}$ such that ${\delta_{B_f|X}= 2}$ cannot be a point (otherwise $X$ would be singular at $F_{B_f|X}$), it remains to study when $F$ consists of an integral surface $F_1$ and an integral curve $C$.
In this case $B_f$ is the family of lines intersecting both $F_1$ and $C$, and hence $X$ must contain the whole fundamental locus.
Following \cite[Theorem 1]{D3} there are seven possible choices for $F_1$ and $C$.
In two of them $F_1$ is a plane and in another $F_1$ is a projected Veronese surface of degree $4$, hence these cases cannot occur.
In other two cases, $F_1$ is a non-degenerate rational scroll possessing a unisecant plane curve.
Then the remaining possibilities give rise to items (1) and (2) in the statement of Theorem \ref{theorem MAIN}.

Finally, we have to prove that the general smooth threefold ${X\subset \mathbb{P}^4}$ of degree $d\geq 7$ has $\degree_r(X)=d-1$, and if $d\geq 8$, any dominant rational map ${X\dashrightarrow \mathbb{P}^3}$ of degree $d-1$ is the projection from a point ${Q\in X}$.
To this aim we recall that $X$ does not contain rational curves (cf. \cite{Cl,Vo}) as $X$ is general of degree ${d\geq 7}$.
Thus the first part of the proof assures that $X$ satisfies neither (1) nor (2), and hence $\degree_r(X)=d-1$.

Then let us assume ${d\geq 8}$ and let ${f\colon X\dashrightarrow \mathbb{P}^3}$ be a dominant rational map of degree $d-1$.
By Theorem \ref{theorem MAPS AND CONGRUENCES}, any such a map still induces a first order congruence of lines $B_f$ in $\mathbb{P}^4$.
When the fundamental locus $F$ of $B_f$ is not a point, any irreducible component of $F$ is covered by rational curves (cf. Remark \ref{remark RATIONAL CURVES}).
Therefore none of them lies on $F_{B_f|X}$.
Since ${\deg f=d-1=d-\delta_{B_f|X}}$, we deduce that $F=F_{B_f|X}$ is a point, $B_f$ is the star of lines through it, and $f$ is the induced projection.
\end{proof}

\begin{example}\label{example P^4 LINE and SURFACE}
Let ${X\subset \mathbb{P}^4}$ be a smooth threefold of degree ${d\geq 7}$ as in (1) of Theorem \ref{theorem MAIN}, so that $X$ contains a rational surface $S$ of degree $s$ and a $(s-1)$-secant line $\ell$ of $S$.
Then we can easily define a dominant rational map $f\colon X\dashrightarrow \mathbb{P}^3$ of degree $d-2$ following the argument of Example \ref{example P^3 LINE and CURVE}.
Namely, given a general point ${P\in X}$ we consider the two-dimensional plane $H$ through $P$ containing $\ell$.
Such a plane cuts out on $X$ a reducible curve consisting of $\ell$ and a degree $d-1$ curve $D_H$.
Moreover, there exists a unique point ${P_H\in D_H\cap S}$ not lying on $\ell$, and we may consider the degree $d-2$ map ${f_H\colon D_H\longrightarrow\mathbb{P}^1}$ projecting $D_H$ from such a point.
Since the planes containing $\ell$ describe a rational surface ${B_\ell\subset\mathbb{G}(2,4)}$, we can argue as in Example \ref{example P^3 LINE and CURVE} ad we can define $f$.
\end{example}

\begin{example}\label{example P^4 LINE in a SURFACE}
Let ${X\subset \mathbb{P}^4}$ be a hypersurface containing a rational surface $S$ of degree $s$ and a line $\ell\subset S$ as in (2) of Theorem \ref{theorem MAIN}.
Up to equivalence, we can then define a dominant rational map $f\colon X\dashrightarrow \mathbb{P}^3$ of degree $d-2$ as follows.
Let us fix a general plane ${\pi\subset \mathbb{P}^4}$ containing $\ell$ and for the general point $P\in X$, let us consider the hyperplane $H$ through $P$ containing $\pi$.
Then $H\cap X$ consists of a surface $\Sigma_H$ of degree $d$ containing a rational curve $C_H$ of degree $s-1$ and the line $\ell$ which is $(s-2)$-secant to $C_H$ by assumptions.
Hence we can define a dominant rational map $\Sigma_H\dashrightarrow \mathbb{P}^2$ as in Example \ref{example P^3 LINE and CURVE}.
By varying $H$ in the pencil ${B_\pi\subset \mathbb{G}(3,4)}$ of hyperplanes containing $\pi$, we then obtain a dominant rational map ${X\dashrightarrow B_\pi\times \mathbb{P}^2}$ of degree $d-2$.
\end{example}

\begin{remark}
Let ${X\subset \mathbb{P}^n}$ be a smooth hypersurface of degree $d$.
When ${n=2k+1}$ is odd, it is possible to provide examples with ${\degree_r(X)= d-2}$ for any $d\geq 3$.
For instance we can extend Example \ref{example P^3 LINE and CURVE} and consider the smooth hypersurface defined by the equation ${x_0^d-x_0x_1^{d-1}+\ldots+x_{2k}^d-x_{2k}x_{2k+1}^{d-1}=0}$.
Since $X$ contains the $k$-planes ${x_0=\ldots=x_{2k}=0}$ and ${x_0-x_1=\ldots=x_{2k}-x_{2k+1}=0}$, and the family of lines meeting both is a first order congruence of lines of $\mathbb{P}^{2k+1}$, we can easily define a dominant rational map ${X\dashrightarrow \mathbb{P}^{2k}}$ of degree $d-2$.\\
On the other hand, when ${n=2k}$ is even, we do not know any explicit example of such an $X$ with ${\degree_r(X)\leq d-2}$.
In particular, it is not possible to argue as above and to use linear subvarieties to produce dominant rational maps.
In order to have ${\degree_r(X)\leq d-2}$, the hypersurface $X$ should indeed contain at least one subvariety of dimension $\geq k$ (cf. \cite[Theorem 5]{D2}), but it cannot be linear as $X$ is assumed to be nonsingular.
\end{remark}

\section*{acknowledgements}

We are extremely grateful to Gian Pietro Pirola, who made us interested in this problem and patiently assisted us with many suggestions.
We would like to thank C. Ciliberto, A. F. Lopez and E. Schlesinger for helpful discussions.
We would also like to thank P. Frediani and L. Gatto for their support.

\bibliographystyle{amsplain}

\end{document}